\documentclass[11pt]{article}
\usepackage{amsmath,amssymb,amsthm}
\usepackage[margin=1in]{geometry}
\newtheorem{theorem}{Theorem}

\newtheorem{lemma}[theorem]{Lemma}

\theoremstyle{definition}

\usepackage{tikz}
\usepackage{float}
\usepackage[colorlinks=true, citecolor=green, linkcolor=blue, urlcolor=blue]{hyperref}
\usepackage{titling}
\usepackage{comment}

\newenvironment{mainthm}[1]
  {\trivlist\item[\hskip\labelsep{\bfseries #1.}]\itshape\ignorespaces}
  {\endtrivlist}

\title{Regular Graphs Without Isomorphic Bisections}
\author{Elie Feinsilber\\
{\small\textit{efeinsilber@wesleyan.edu}}}
\date{}
\begin{document}
\maketitle

\begin{abstract}
A graph is bisectable if its vertex set can be split into two sets inducing isomorphic
subgraphs. Ando conjectured that every cubic graph is bisectable, and Das, Pokrovskiy, and Sudakov proved it for large enough connected cubic graphs. We give a constructive proof that, for every $d\geq5$, there exists a non-bisectable $d$-regular graph of even order. 
\end{abstract}

\noindent\textbf{Keywords:}
Graphs; regular graphs; isomorphic bisections; Ando's conjecture;
graph partitions. 
\begingroup
\renewcommand{\thefootnote}{}
\footnotemark\footnotetext{\textbf{2020 MSC:} 05C60 (Primary); 05C07 (Secondary), 05C15.}
\endgroup
\section{Introduction}

Partitions of a graph into isomorphic pieces have a long history, but mostly for edges. Isomorphic factorisations, i.e. partitions of $E(G)$ into pairwise isomorphic spanning subgraphs,
were studied systematically by Harary, Robinson and Wormald in a series of papers beginning with \cite{HRW}. 
Jacobson and Sch\"onheim separately asked for the largest $s$ such that every graph with $m$ edges contains two
edge-disjoint isomorphic subgraphs with $s$ edges; after bounds of Erd\H{o}s, Pach and Pyber
\cite{EPP}, Lee, Loh and Sudakov \cite{LLS} showed that $s=\Theta\big((m\log m)^{2/3}\big)$.

In this paper, we consider the \emph{isomorphic bisection} of a graph $G$: a triple $(A,B,\varphi)$ with $V(G)=A\sqcup B$ and $\varphi$ an isomorphism from the subgraph induced by the vertex set $A$, $G[A]$, onto $G[B]$. We say $G$ is \emph{bisectable} if it has an isomorphic bisection. We take all graphs to be finite, simple, and undirected, and use ``bisection" as vocabulary shortcut for ``isomorphic bisection". 
For cubic graphs, bisectability is Ando's conjecture, made in the 1990s and first stated in print in~\cite{AGLM}. It holds for $3$-edge-colourable
cubic graphs: two colour classes form a $2$-factor of even cycles, and colouring each cycle alternately gives two halves that induce matchings of the same size \cite{BL,AGLM}. Abreu, Goedgebeur, Labbate and Mazzuoccolo~\cite{AGLM} verified the conjecture by computer for all cubic graphs with at most $32$ vertices, and Das, Pokrovskiy and Sudakov~\cite{DPS} proved it for all sufficiently large connected cubic graphs. 

We show that the conjecture does not extend to $d$-regular graphs with $d\geq5$. It is clear that a bisectable graph must have even order, so for $d$ even, $K_{d+1}$ is directly not bisectable for parity reasons alone. We go further and prove the theorem below using a construction which attaches a dense block to a circulant graph, or a graph of girth at least $5$ depending on $d$. 

\begin{mainthm}{Theorem~\ref{thm:main}}
For all $d\geq 5$, there exists a connected $d$-regular graph of even order that is not bisectable. 
\end{mainthm}

More precisely, for all $d\ge5$ and every even $n\ge 2d^2+200$ we show that there is a connected $d$-regular graph of order $n$ that is not bisectable. It suffices that $n\ge4d+8$ if $d$ is even, $n\ge2d^2-d+3$ if $d\ge7$ is odd, and $n\ge208$ if $d=5$. 

The construction begins with a dense block of odd order, which any partition of the vertices must split unevenly. Lemma \ref{lem:gadget} shows that this prevents an isomorphic bisection, provided every triangle lies inside the block and, when \(d=5\), every induced 4-cycle does as well. We then attach an auxiliary graph to obtain a connected \(d\)-regular graph of even order while preserving these conditions. For \(d\ge6\), triangle-free circulants suffice (Lemma \ref{lem:circulant}). For \(d=5\), we use a two-layer graph of girth at least 5 (Lemma \ref{lem:girth5}). For the proof below, remember Mantel’s Theorem \cite{Mantel}: a graph on $n$ vertices with no triangle contains at most $n^2/4$ edges.

\section{Results}

For even $d$ let $W_d=K_{d+1}-uv$, with $u,v$ of degree $d-1$. For odd $d$ let
$W_d=K_{d+2}-F_d$, where $F_d$ is the disjoint union of a path $qpz$ and a set $M_d$ of $(d-1)/2$
further disjoint edges, with $p$ of degree $d-1$. All other vertices have degree
$d$ in $W_d$, and $|V(W_d)|$ is odd.

\begin{lemma}\label{lem:gadget}
Let $d\ge5$ and let $G$ be a $d$-regular graph of even order with a set $W$ of vertices such that
$G[W]\cong W_d$. Suppose that every triangle of $G$ lies in $W$, and, if $d=5$, that every induced
$4$-cycle of $G$ lies in $W$. Then $G$ is not bisectable.
\end{lemma}

\begin{proof}
Let $F$ be $F_d$ if $d$ is odd, and the single edge $uv$ if $d$ is even. In either case $G[W]=K_W-F$, where $F$ has
maximum degree at most $2$ and contains no triangle. For $S\subseteq V(G)$ let $\tau(S)$ be the set of
vertices lying in a triangle of $G[S]$; by assumption $\tau(S)=\tau(S\cap W)$. 

Remark that if $S\subseteq W$ and $|S|\ge6$, or $d$ is even and $|S|\ge4$, then $\tau(S)=S$. 
Indeed, if $|S|\ge6$, each vertex in $S$ has at least $|S|-3\ge3$ neighbours in $S$ since the ``worst" case for the $F$ removed is the disjoint union of a two edge path and $(d-1)/2$ disjoint edges. Two of any three of these neighbours are
adjacent because $F$ is triangle-free and we are in a complete graph on at least $6$ vertices. If $d$ is even, $G[S]$ is $K_S$ or $K_S$ minus an edge. Thus every vertex of $S$ lies in a triangle. 

Now towards a contradiction, let $(A,B,\varphi)$ be a bisection. So $\varphi$ maps the triangles of $G[A]$ to those of $G[B]$ and $|\tau(A\cap W)|=|\tau(B\cap W)|$. Let $S = A \cap W$ and $k = |S|$. Since $|W|$ is odd by construction, we may assume $k=|A\cap W|>|B\cap W|$. If
$\tau(A\cap W)=A\cap W$ then $|\tau(B\cap W)|\le|B\cap W|<k$, a contradiction. This
settles even $d \geq 6$: $|V(W_d)| = d+1$ so $k = |S|\ge(d+2)/2\ge4$. As well as odd $d\ge9$ (where $k\ge(d+3)/2\ge6$); and leaves
$k=5$ for $d=7$, and $k\in\{4,5\}$ for $d=5$. 

\medskip

First, for $d=7$, $k=5$. Let $S=A\cap W$. If $|\tau(S)|=5$ we are done since $|B\cap W| = 4$. So let $v\in S\setminus\tau(S)$. It has at least $5-1-2=2$
neighbours in $S$, so it must have exactly two (else a triangle would form), they are joined in $F$, and
$v$ has two $F$-neighbours in $S$. Hence $(v=p), q,z\in S$, and the other two vertices of $S$ form an
edge of $M_7$. So either $|\tau(S)|=5>|B\cap W|$, contradicting
$|\tau(S)|=|\tau(B\cap W)|$, or
$S=\{p,q,z,r,r'\}$ with $rr'\in M_7$.
In the latter case, both $\{q,z,r\}$ and $\{q,z,r'\}$ are triangles,
while $p$ lies in none, so $\tau(S)=\{q,z,r,r'\}$.
Meanwhile, $B\cap W$ consists of the endpoints of the other two edges
of $M_7$. Its induced subgraph is therefore $K_4$ with two disjoint
edges deleted, i.e. a $4$-cycle. Thus $\tau(B\cap W)=\emptyset$,
again contradicting $|\tau(S)|=|\tau(B\cap W)|$ (see Figure \ref{fig:d7}). 


Now for $d=5$ and $k=5$, $k=|A\cap W|= 5$ implies $|B\cap W|=2$, so we would need $\tau(A\cap W)=\emptyset$ for the bisection. The four edges of $F_5$ cover all
seven vertices of $W$, so at most three of them lie in the five vertices of $A\cap W$, and $G[A\cap W]$
is a graph on five vertices with at least $7$ edges, thus not triangle-free by Mantel's theorem. 

Finally, let $d=5$, $k=4$. Any four vertices of $W$ contain at most two edges of $F_5$, so $G[A\cap W]$ is $K_4$,
$K_4-e$, $C_4$, or a triangle with a pendant edge (when three of the four vertices of $A\cap W$ are $\{p,q,z\}$). In these cases, $|\tau(A\cap W)|$ is $4,4,0,3$ respectively. $|\tau(A\cap W)| = 4$ would be a direct contradiction. If $|\tau(A\cap W)| = 3$, then it is the triangle with pendant edge, which makes $|\tau(B\cap W)| = 3$ impossible: $p, q, z \in S$ so whichever three of the remaining four vertices are in $B\cap W$ cannot induce a triangle. Finally if $S$ is triangle-free, it contains a $4$-cycle (by hypothesis, all of which are in $W$). So we cannot have an isomorphism with the three vertices of $B \cap W$. 
\end{proof}

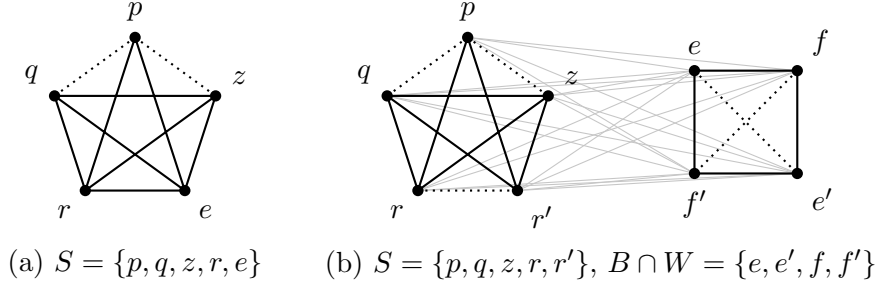
\begin{figure}[H]
\centering
\begin{tikzpicture}[scale=0.8,
  v/.style={circle,fill=black,inner sep=1.5pt},
  g/.style={thick},
  f/.style={thick,dotted},
  cr/.style={gray!45},
  c4/.style={line width=2pt}]
\begin{scope}
  \coordinate (p) at (90:1.4);   \coordinate (z) at (18:1.4);  \coordinate (e) at (-54:1.4);
  \coordinate (r) at (-126:1.4); \coordinate (q) at (162:1.4);
  \draw[g] (z)--(e)--(r)--(q) (p)--(e) (p)--(r) (q)--(z) (q)--(e) (z)--(r);
  \draw[f] (p)--(q) (p)--(z);
  \node[v,label=90:$p$] at (p) {};   \node[v,label=18:$z$] at (z) {};
  \node[v,label=-54:$e$] at (e) {};  \node[v,label=-126:$r$] at (r) {};
  \node[v,label=162:$q$] at (q) {};
  \node at (0,-2.35) {(a) $S=\{p,q,z,r,e\}$};
\end{scope}
\begin{scope}[xshift=5.5cm]
  \coordinate (p) at (90:1.4);   \coordinate (z) at (18:1.4);  \coordinate (r2) at (-54:1.4);
  \coordinate (r) at (-126:1.4); \coordinate (q) at (162:1.4);
  \coordinate (e) at (3.75,0.85);  \coordinate (f) at (5.45,0.85);
  \coordinate (e2) at (5.45,-0.85); \coordinate (f2) at (3.75,-0.85);
  \foreach \a in {p,q,z,r,r2} \foreach \b in {e,f,e2,f2} \draw[cr] (\a)--(\b);
  \draw[g] (q)--(z) (q)--(r) (q)--(r2) (z)--(r) (z)--(r2) (p)--(r) (p)--(r2);
  \draw[f] (p)--(q) (p)--(z) (r)--(r2);
  \draw[g] (e)--(f)--(e2)--(f2)--cycle;
  \draw[f] (e)--(e2) (f)--(f2);
  \node[v,label=90:$p$] at (p) {};    \node[v,label=18:$z$] at (z) {};
  \node[v,label=-54:$r'$] at (r2) {}; \node[v,label=-126:$r$] at (r) {};
  \node[v,label=162:$q$] at (q) {};
  \node[v,label=90:$e$] at (e) {};     \node[v,label=45:$f$] at (f) {};
  \node[v,label=-45:$e'$] at (e2) {};  \node[v,label=-90:$f'$] at (f2) {};
  \node at (2.2,-2.35) {(b) $S=\{p,q,z,r,r'\}$, $B\cap W=\{e,e',f,f'\}$};
\end{scope}
\end{tikzpicture}
\caption{The case $d=7$, $k=5$ in the proof of Lemma~\ref{lem:gadget}, with $M_7=\{rr',ee',ff'\}$.
Solid lines are edges of $G$, dotted lines are edges of $F_7$ (non-edges of $G$). Cases show the contradictions depending on whether the neighbours of $p$ in $S$ are adjacent to each other in $S$ or in $F_7$.}
\label{fig:d7}
\end{figure}


Now we give a lemma that proves key to the construction of what will sit outside of $W_d$. 
For $m\ge3$ and $S\subseteq\mathbb Z_m\setminus\{0\}$ with $S=-S$, the circulant $C_m(S)$ has vertex set
$\mathbb Z_m$ and $x\sim y$ iff $y-x\in S$; it is $|S|$-regular. Let
$I_m=\{x\in\mathbb Z: m/3<x<2m/3\}$, read modulo $m$.

\begin{lemma}\label{lem:circulant}
If $S\subseteq I_m$ and $S=-S$, then $C_m(S)$ is triangle-free. If moreover $S$ contains two
consecutive integers, $C_m(S)$ is connected.
\end{lemma}

\begin{proof}
$I_m$ is sum-free modulo $m$: for $x,y\in I_m$ we have $2m/3<x+y<4m/3$, so the residue of $x+y$ lies in
$(2m/3,m)\cup[0,m/3)$. A triangle $x,y,z$ would give $s_1+s_2+s_3=0$ with $s_1=y-x$, $s_2=z-y$,
$s_3=x-z\in S$, so $s_1+s_2=-s_3\in S$, which is impossible. If $s,s+1\in S$, the subgroup generated by
$S$ contains $1$.
\end{proof}

Before proving our main result, we give below the construction that will be necessary for $d=5$. 

\begin{lemma}\label{lem:girth5}
For $m\ge100$ let $B_m$ be the graph on $\mathbb Z_m\times\{0,1\}$ with edges
\[
  (x,0)(x+1,0),\qquad (x,1)(x+20,1),\qquad (x,0)(x+t,1)\ \ (t\in T=\{0,5,12\}).
\]
Then $B_m$ is connected, $5$-regular, of order $2m$ and girth at least $5$.
\end{lemma}

\begin{proof}
Each vertex has two neighbours in its own layer (as $m>40$) and three in the other. For example, for some $(x, 0)$ in layer $0$, it has neighbours $(x+1, 0)$ and $(x-1, 0)$ in that layer, and $(x+t, 1)$ for $t \in \{0, 5, 12\}$ in the other. The first layer (layer $0$) is a Hamiltonian cycle, and every vertex of the second layer is joined to it (per our third rule with $t=0$). Thus $B_m$ is connected and $5$-regular. 

Now, towards a contradiction, assume there exists a cycle $\ell$ of length at most $4$. There are four kind of ways to walk through this cycle, and we can record the coordinates changes at each step: $\pm1$ inside layer $0$, $\pm20$ inside layer $1$, $+t$ from layer $0$ to layer $1$ and $-t$ back, with $t\in T$. Since we must return to the vertex we began with, the sum of coordinate changes must be $0 \mod m$. And since the cycle has at most four steps, each with coordinate change at most $20$, $4\cdot20 < m$ so the sum of coordinate changes must be exactly $0$. Remark that consecutive vertices of a cycle are distinct and no vertex repeats, so two consecutive
steps inside the same layer have the same sign, and a step $+t$ followed by $-t'$ (or $-t$ followed by
$+t'$) has $t\ne t'$. Note also that we must change layer an even amount of times: $0, 2, \text{ or }4$ times. 

Per construction of $B_m$, combinations of $t \in T$ have results in $T- T\setminus\{0\} = \{\pm5,\pm7,\pm12\}$. We can therefore see whether a cycle can occur based on the number of changes of layer. If there are none, the sum of the steps of size $1$ or $20$ without repeats does not amount to $0$. If we change layer twice, for a $3$-cycle, with say intra-layer step $s$ and traversing steps $\pm t$, we would need $s+t-t' = 0$. Since $s\in\{\pm1,\pm20\}$ and we know $T- T\setminus\{0\} = \{\pm5,\pm7,\pm12\}$, this is impossible. For a $4$-cycle, depending on whether the intra-layer steps happen in the same layer or a different one, we would need $t'-t\in\{\pm2,\pm40\}$ or $\{\pm19,\pm21\}$, again impossible. Finally, with four changes of layer, we would need $t_1-t_2+t_3-t_4=0$ (consecutively distinct). The six possible sums $t+t'$ are $0,5,10,12,17,24$, all distinct, so $\{t_1,t_3\}=\{t_2,t_4\}$ as multisets, which implies $t_1=t_2$ or
$t_1=t_4$.

In each case we get a contradiction, so $B_m$ has no cycle $\ell$ of length inferior or equal to $4$. 
\end{proof}

We can now prove the main result. 

\begin{theorem}\label{thm:main}
    For all $d\geq 5$, there exists a connected $d$-regular graph of even order that is not bisectable.
\end{theorem}

\begin{proof}
For $d\ge6$ even and even $n\ge4d+8$, let $m=n-d-1$, which is odd and at least $3d+7$. The open interval
$(m/3,m/2)$ contains at least $m/6-1\ge d/2$ consecutive integers; let $J$ be a set of $d/2$ consecutive of them
and $S=J\cup(-J)\subseteq I_m$. By Lemma~\ref{lem:circulant}, $C_m(S)$ is a connected triangle-free
$d$-regular graph. Delete an edge $xy$ of $C_m(S)$, add $W_d$, and add the edges $ux$ and $vy$ (where $uv$ was a non-edge per our construction of $W_d$ in the case where $d$ is even). The result
$G$ is $d$-regular of order $n$, and connected, since $x$ and $y$ are joined through $W_d$. Every
triangle of $G$ lies in $W_d$: $C_m(S)$ is triangle-free, and a triangle meeting both parts would use $ux$
or $vy$, but $u$ and $x$ have no common neighbour, and neither have $v$ and $y$. Thus, by
Lemma~\ref{lem:gadget}, $G$ is not bisectable.

Now for $d\ge7$ odd and even $n\ge2d^2-d+3$, let $m=n-d-3$, which is even and at least $2d(d-1)\ge3d+3$. Let
$J$ be a set of $(d-1)/2$ consecutive integers in $(m/3,m/2)$ and $S=\{m/2\}\cup J\cup(-J)\subseteq I_m$;
then $C_m(S)$ is a connected triangle-free $d$-regular graph. It has an induced matching
$x_1y_1,\dots,x_sy_s$ with $s=(d-1)/2$. Indeed, choosing edges greedily, each choice $xy$ excludes all edges
meeting $N[x]\cup N[y]$, which is at most $2d^2$. And $C_m(S)$ has $dm/2 \geq d \cdot 2d(d-1)/2 \geq 2d^2 \cdot (d-1)/2 = 2d^2s$ edges, so the matching exists. Now delete these $s$ edges,
add a vertex $w$ adjacent to $x_1,y_1,\dots,x_s,y_s$, add $W_d$ and the edge $pw$. The result $G$ is
$d$-regular of order $(d+2)+(m+1)=n$ and connected (each deleted edge $x_iy_i$ is replaced by the path
$x_iwy_i$). The neighbours of $w$ are pairwise non-adjacent, since the matching is induced and its edges
were deleted, and $p,w$ have no common neighbour; so every triangle of $G$ lies in $W_d$, and
Lemma~\ref{lem:gadget} applies.

Finally, for $d=5$ and even $n\ge208$, let $m=(n-8)/2\ge100$ and $h=\lfloor m/2\rfloor$. In $B_m$ of Lemma \ref{lem:girth5}, the edges
$e_1=(0,0)(1,0)$ and $e_2=(h,0)(h+1,0)$ are at distance at least $3$: each step changes the first
coordinate by at most $20$, and their first coordinates are at circular distance at least $m/2-2\ge48$.
Delete $e_1,e_2$ and add a vertex $w$ adjacent to their four ends; call the resulting graph $H$. A cycle through
$w$ leaves and re-enters $w$ at two of these ends: if they are the ends of the same $e_i$, their distance in
$B_m-e_i$ is at least $4$ because $B_m$ has girth at least $5$; otherwise it is at least $3$. So $H$ has
girth at least $5$; it is connected of order $2m+1$, and all its degrees are $5$ except $w$ which has $\deg (w)=4$. Add
$W_5$ and the edge $pw$. The resulting graph $G$ is connected, $5$-regular, of order $7+2m+1=n$. Since $pw$ is a
bridge, every cycle of length at most $4$ lies in $W_5$ or in $H$, hence in $W_5$, and
Lemma~\ref{lem:gadget} applies again.

Finally $4d+8$, $2d^2-d+3$ and $208$ are at most $2d^2+200$, hence the bound mentioned in the introduction. 
\end{proof}

\newpage

\section*{Funding}

No specific funding was received for this work.

\section*{Declaration of competing interests}

The author has no competing interests to disclose. 

\end{document}